\documentclass[12pt,reqno,a4paper]{amsart} 
\usepackage{graphicx}
\usepackage{amssymb,amsmath}
\usepackage{amsthm}
\usepackage{color,graphicx}
\usepackage{hyperref}
\usepackage{color}
\usepackage{mathabx}
\usepackage{epstopdf}

\usepackage{subfigure} 

\usepackage{verbatim}

\newcommand{\be}{\begin{equation}}

\newcommand{\ee}{\end{equation}}

\newtheorem{prop}{Proposition}[section]

\newcommand{\sn}{{\rm \,sn}}
\newcommand{\dn}{{\rm \,dn}}

\newcommand{\Ker}{{\rm \,Ker}}

\numberwithin{equation}{section}
\numberwithin{figure}{section}

\newtheorem{theorem}{Theorem}[section]
\newtheorem{proposition}[theorem]{Proposition}
\newtheorem{remark}[theorem]{Remark}
\newtheorem{lemma}[theorem]{Lemma}

\newtheorem{definition}[theorem]{Definition}

\begin{document}
\vglue-1cm \hskip1cm
\title[Zero-mean periodic waves for the $\phi^4-$equation]{On the orbital stability of periodic snoidal waves for the $\phi^4-$equation}

\begin{center}

\thanks{B.S. Lonardoni is supported by CAPES/Brazil - Finance Code 001. F. Natali is partially supported by CNPq/Brazil (grant 303907/2021-5).}

\maketitle

{\bf Beatriz Signori Lonardoni}

{Departamento de Matem\'atica - Universidade Estadual de Maringá\\
	Avenida Colombo, 5790, CEP 87020-900, Maring\'a, PR, Brazil.}\\
{ beatrizslonardoni@gmail.com}

{\bf F\'abio Natali}

{Departamento de Matem\'atica - Universidade Estadual de Maring\'a\\
	Avenida Colombo, 5790, CEP 87020-900, Maring\'a, PR, Brazil.}\\
{ fmanatali@uem.br}

\vspace{3mm}

\end{center}

\begin{abstract}

The main purpose of this paper is to investigate the global well-posedness and orbital stability of odd periodic traveling waves for the $\phi^4-$equation in the Sobolev space of periodic functions with zero mean. We establish new results on the global well-posedness of weak solutions by combining a semigroup approach with energy estimates. As a consequence, we prove the orbital stability of odd periodic waves by applying a Morse index theorem to the constrained linearized operator defined in the Sobolev space with the zero-mean property.

\end{abstract}

\section{Introduction} 

Consider the well-known $\phi^4-$equation
\begin{equation}\label{KF2}
\phi_{tt}-\phi_{xx}-\phi+\phi^3=0,
\end{equation}
where $\phi:\mathbb{R} \times \mathbb{R}_+ \rightarrow \mathbb{R}$ is an $L$-periodic function at the spatial variable. This means that we have $\phi(x+L,t)=\phi(x,t)$ for all $t\geq0$. In a convenient scenario, equation $(\ref{KF2})$ is a typical Klein-Gordon equation with non-negative energy and it plays an important role in nuclear and particle physics. From a mathematical point of view, the $\phi^4-$equation supports kink and anti-kink solutions.  An important feature of these waves is that they are stable, localized solutions that model domain walls, phase transitions, and nonlinear wave propagation.\\
\indent Equation \eqref{KF2} has an abstract Hamiltonian system form\begin{equation}\label{hamilt31}
	\frac{d}{dt}\Phi(t)=J \mathcal{E}'(\Phi(t)),
\end{equation}   where $\Phi=(\phi,\phi_t)$ and $J$ is given by
\begin{equation}\label{J}
J=\left(
\begin{array}{cc}
 0 & 1\\
-1 & 0\end{array}
\right).
\end{equation}
 Now, if $Z=H^1_{per}\times L^2_{per} $, we see that $\mathcal{E}':Z\rightarrow Z'$ denotes the Fr\'echet derivative of the conserved quantity (energy) $\mathcal{E}:Z\rightarrow \mathbb{R} $ given by
\begin{equation}\label{E}
\mathcal{E}(\phi, \phi_t)
=\displaystyle\frac{1}{2}\int_{0}^{L} \left[\phi_x^2+\phi_t^2-\phi^2+\frac{\phi^4}{2}\right] \;dx.
\end{equation}
 Moreover, \eqref{KF2} has another conserved quantity defined in $Z$ given by 
\begin{equation}\label{F}
	\mathcal{F}(\phi,\phi_t)=\displaystyle\int_{0}^{L}\phi_x \phi_t \; dx.
\end{equation}

A fundamental property associated with the equation \eqref{KF2} is the existence of kink, anti-kink and periodic traveling wave solutions of the form
\begin{equation}\label{PW}
\phi(x,t)=h(x-ct),
\end{equation}
where $c \in\mathbb{R}$ represents the wave speed and $h=h_c: \mathbb{R}\rightarrow \mathbb{R}$ is an $L$-periodic  smooth function.

\indent In our paper, we consider the case where the solution $h$ is odd. In fact, substituting \eqref{PW} into \eqref{KF2}, it follows that $h$ satisfies the following ODE 
\begin{equation}\label{KF3}
-\omega h''-h+h^3=0,
\end{equation}
where $\omega=1-c^2$ is assumed to be non-negative, which implies $c\in(-1,1)$. First, we have the existence of kink solution associated with the equation $(\ref{KF3})$, given by 
\begin{equation}\label{Sol1}
	h(x)= \tanh  \left(\frac{x}{\sqrt{2\omega}}\right).
\end{equation} 
The anti-kink solution is given by $h(x)=-\tanh  \left(\frac{x}{\sqrt{2\omega}}\right)$. In the periodic context, one can find an explicit solution depending on the Jacobi elliptic function of snoidal type as
\begin{equation}\label{Sol2}
h(x)=\frac{\sqrt{2k}}{\sqrt{k^2+1}} \sn \left(\frac{4K(k)}{L}\;x;k \right),
\end{equation} 
where $k\in\left(0,1\right)$ is called the modulus of the elliptic function and $$K(k)=\displaystyle\int_{0}^{\frac{\pi}{2}}\frac{1}{\sqrt{1-k^2\sin^2(\theta)}}d\theta$$ is the complete  elliptic integral of the first kind.

The value of $\omega$ depends on $k$ and $L$ and it is expressed by    \begin{equation}\label{ccn}
	\frac{1}{\sqrt{\omega}}=\frac{4K(k)\sqrt{1+k^2}}{L}.\end{equation}
By assuming that $0<L<2\pi$, we obtain from \eqref{ccn} that $0<\omega<1$ and the modulus $k$ varies over the  open interval $\left(0,1\right)$. It is important to mention that the periodic wave in $(\ref{Sol2})$ is odd and, therefore, possesses the zero-mean property. In addition, supposing that $\phi \in H_{per,m}^1$ (the space of functions in the Sobolev space $H_{per}^1$ with the zero-mean property), the condition $0 < L < 2\pi$ also implies, via the Poincaré–Wirtinger inequality 
\begin{equation}\label{PW1}\int_0^L\phi^2dx\leq 	\biggl(\frac{L}{2\pi}\biggl)^2\int_0^L\phi_x^2dx,\end{equation} that the energy $\mathcal{E}$ in $(\ref{E})$ satisfies $\mathcal{E}(\phi, \phi_t) \geq 0$ for all  $t\geq0$.\\
\indent Let us discuss some contributors with respect to the stability of periodic waves for the equation $(\ref{KF2})$ and related topics. In fact, regarding the general equation
\begin{equation}\phi_{tt}-\phi_{xx}+V'(\phi)=0,\label{genKG}\end{equation}
some results concerning spectral/modulational stability of periodic waves have been determined in \cite{jones2} and \cite{jones1} under the condition that $V:\mathbb{R}\rightarrow \mathbb{R}$ is a  periodic (and bounded) nonlinearity (both references include the case $V(u)=\cos(u)$ - the well known sine-Gordon equation). Using assumptions similar to those in \cite{jones1} and \cite{MM}, the authors introduced a concise criterion for the presence of dynamical Hamiltonian-Hopf instabilities, which serves as a practical tool for determining the spectral stability of periodic traveling waves. Additional references on related topics can be found in \cite{pava}, \cite{clarke}, \cite{hakkaev1}, and \cite{stan}. It is important to highlight that the orbital instability of the sine-Gordon equation was established in \cite{Natali2011} in the entire energy space $H_{per}^1\times L_{per}^2$. To this end, the author used the abstract theory in \cite{grillakis1}. Using \cite{grillakis1}, orbital stability results for the sine-Gordon equation, in the space $H_{per,m}^1\times L_{per,m}^2$ consisting of functions in $H_{per}^1\times L_{per}^2$ with the zero-mean property, were established in \cite{nataliSG}. \\
\indent Orbital instability of periodic waves for the model $(\ref{KF2})$ has been determined  in \cite{loreno} and \cite{palacios}, where the authors also used the abstract theory in \cite{grillakis1} adapted to the periodic context. In addition, both authors also studied the orbital stability in the Sobolev space $H_{per,odd}^1\times L_{per,odd}^2$, consisting of odd periodic functions. A generalization of the results in \cite{loreno} and \cite{palacios}, which were obtained for power-type nonlinearities, can be found in \cite{chenpala}.\\
\indent One of the most important features of our work is that we prove orbital stability in the space $Y=H_{per,m}^1\times L_{per,m}^2$, which lies between $H_{per,odd}^1\times L_{per,odd}^2$ (associated with stable waves) and the full space $H_{per}^1\times L_{per}^2$ (associated with unstable waves). A key advantage is that, to study the orbital stability in $H_{per,odd}^1\times L_{per,odd}^2$, one must restrict to stationary waves of the form $\phi(x,t)=h(x)$, since the translational waves $\phi(x,t)=h(x-ct)$ with wave speed $c\in \mathbb{R}$ are not invariant in the space $H_{per,odd}^1\times L_{per,odd}^2$. As far as we can see, this imposes a significant restriction on the analysis of orbital stability for periodic waves.\\
\indent In order to prove our orbital stability in the space $Y$, it is necessary to present some key elements. To begin with, defining
$\mathcal{G}(\phi, \phi_t)=\mathcal{E}(\phi, \phi_t)-c\mathcal{F}(\phi, \phi_t),$ it is clear that any solution of \eqref{KF3} satisfies $\mathcal{G}'(h,ch')=0$, that is, $(h,ch')$ is a critical point of $\mathcal{G}$. We initiate our discussion by considering the assumption that the linearized operator
\begin{equation}\displaystyle\label{opconstrained2}\mathcal{L}_{\Pi}= \mathcal{G}''(h,ch')-\displaystyle \left(
\displaystyle	\begin{array}{ccc}
		\frac{3}{L}\int_{0}^{L} h^2 \cdot \; dx& & 0 \\\\
		0  & & 0
	\end{array}\right)=\displaystyle\mathcal{L}-\displaystyle \left(
\begin{array}{ccc}
\frac{3}{L}\int_{0}^{L} h^2\cdot \; dx & & 0 \\\\
0  & & 0
\end{array}\right),\end{equation} where $\mathcal{L}$ is given by
	\begin{equation}\label{matrixop313}
		\displaystyle \mathcal{L}=\left(
		\begin{array}{ccc}
		-\partial_x^2-1+3h^2 & &c\partial_x\\\\
		\ \ \ \ -c\partial_x & & 1
		\end{array}\right),
	\end{equation}
	has no negative eigenvalues and zero is a simple eigenvalue associated to the eigenfunction $(h',ch'')$. Based on these facts, we can assert the existence of a constant $C>0$ such that
\begin{equation}\label{positiveL}
(\mathcal{L}(p,q),(p,q))_{\mathbb{L}_{per,m}^2}=(\mathcal{L}_{\Pi}(p,q),(p,q))_{\mathbb{L}_{per,m}^2}\geq C||(p,q)||_{\mathbb{L}_{per,m}^2}^2,
\end{equation}
for all $(p,q)\in H_{per,m}^2\times H_{per,m}^1$ such that $(p,q)\bot (h',ch'')$. As established by stability theory in \cite[Section 4]{Natali2015} (see also \cite{grillakis1}), the coercive condition in $(\ref{positiveL})$ is sufficient to establish that the periodic wave $(h,ch')$ is orbitally stable. In Proposition $\ref{leman1}$, we prove that $\mathcal{L}_{\Pi}$ has no negative eigenvalues and zero is a simple eigenvalue associated with the eigenfunction $(h',ch'')$. Our analysis to prove $(\ref{positiveL})$ follows the arguments presented in \cite[Theorem 5.3.2]{kapitula} and \cite[Theorem 4.1]{pel-book}. We summarize our orbital stability result as follows:

\begin{theorem}[Orbital stability for the $\phi^4-$equation]
	Let $L\in (0,2\pi)$ be fixed. If $c \in (-1,1)$ and $h$ is the periodic solution given by $(\ref{Sol2})$, then the periodic wave $(h,ch')$ is orbitally stable in $Y=H_{per,m}^1\times L_{per,m}^2$.
	\label{stabthm}\end{theorem}

\indent Next, we provide a more detailed description of our well-posedness result for the Cauchy problem associated with the evolution equation $(\ref{KF2})$, and we establish a connection between this result and orbital stability.  Let us consider the well-known Cauchy problem
   \begin{equation}\label{CPKG}
\left\{\begin{array}{llll}
\phi_{tt}-\phi_{xx}-\phi+\phi^3=0, \:\:\:\: \:\text{in} \: [0,L]\times(0,+\infty), \\\\
\phi(x,0)=\phi_0(x),\:\:\:\:\:\:\:\:\:\:\:\:\:\:\:\:\:\:\: \text{in} \: [0,L],\\\\
\phi_t(x,0)= \phi_1(x),\:\:\:\:\:\:\:\:\:\:\:\:\:\:\:\:\:\: \text{in} \: [0,L].
\end{array}\right.
\end{equation}
It is not possible to guarantee, using the standard semigroup approach as in \cite{pazy}, that $(\ref{CPKG})$ is at least locally well-posed in a Sobolev product space $H_{per,m}^s\times H_{per,m}^r$ for a suitable choice of integers $s,r\geq1$. Indeed, using \cite{pazy} we cannot guarantee that the auxiliary Cauchy problem 
$(\ref{CPKG})$ written in matrix form 
\begin{equation}\label{CPKG1}
\begin{cases}
	\left(\begin{array}{lll}\phi\\\beta\end{array}\right)_t
	=\left(\begin{array}{lll}0&  1\\ \partial_x^2&  0\end{array}\right)\left(\begin{array}{lll}\phi\\\beta\end{array}\right)+\left(\begin{array}{ccc}0\\ \phi-\phi^3\end{array}\right),\:\:\:\:\:\:\:\:\:\: \text{in} \: [0,L]\times(0,+\infty),\\\\
	\left(\begin{array}{llll}\phi(0)\\\beta(0)\end{array}\right)=\left(\begin{array}{ccc}\phi_0\\ \beta_0\end{array}\right), \:\:\:\:\:\:\:\:\:\:\:\:\:\:\:\:\:\:\:\:\:\:\:\:\:\:\:\:\:\:\:\:\:\:\:\:\:\:\:\:\:\:\:\:\:\:\:\:\:\:\:\:\:\:\:\: \text{in} \: [0,L],
\end{cases}
\end{equation}
where $\beta=\phi_t$, is locally well-posed in $H_{per,m}^s\times H_{per,m}^r$ for convenient integers $s,r\geq1$. As far as we know, the local well-posedness result in $H_{per,m}^2\times H_{per,m}^1$ is unexpected when employing the standard semigroup approach in the Cauchy problem $(\ref{CPKG1})$ since it is not natural that $\displaystyle\mathcal{H}(\phi,\psi)=\int_0^L\phi(x,t)dx$ be a conserved quantity for all $t>0$. 
 To resolve this challenge, it is necessary to examine the modified  Cauchy problem related to the equation in $(\ref{CPKG})$ expressed by

\begin{equation}\label{CPKG2}
\begin{cases}
	\left(\begin{array}{lll}\phi\\\psi\end{array}\right)_t
	=\left(\begin{array}{lll}\partial_x^{-1}&  0\\ 0&  \partial_x\end{array}\right)\left(\begin{array}{lll}0&  1\\ \partial_x^2&  0\end{array}\right)\left(\begin{array}{lll}\phi\\\psi\end{array}\right)+\left(\begin{array}{ccc}0\\ \partial_x(\phi-\phi^3)\end{array}\right),\:\text{in} \: [0,L]\times(0,+\infty),\\\\
	\left(\begin{array}{lll}\phi(0)\\\psi(0)\end{array}\right)=\left(\begin{array}{ccc}\phi_0\\ \psi_0\end{array}\right), \:\:\:\:\:\:\:\:\:\:\:\:\:\:\:\:\:\:\:\:\:\:\:\:\:\:\:\:\:\:\:\:\:\:\:\:\:\:\:\:\:\:\:\:\:\:\:\:\:\:\:\:\:\:\:\: \:\:\:\:\: \:\:\:\:\: \:\:\:\:\: \:\:\:\:\:\:\:\:  \text{in} \: [0,L],
\end{cases}
\end{equation}
where $\phi_t=\partial_x^{-1}\psi$ and $\partial_x^{-1}:L_{per,m}^2\rightarrow H_{per,m}^1$ is the well-known anti-derivative bounded linear operator defined in $L_{per,m}^2$. If the pair $(\phi,\psi)$ is a smooth solution to the equation in $(\ref{CPKG2})$ in an appropriate space, such as $H_{per,m}^3\times H_{per,m}^1$, we obtain that $\phi$ is a smooth solution of the projected Cauchy problem 
\begin{equation}\label{CPKG3}
\left\{\begin{array}{llll}
\displaystyle\phi_{tt}-\phi_{xx}-\phi+\phi^3-\frac{1}{L}\int_0^L\phi^3dx=0, \:\:\:\: \:\text{in} \: [0,L]\times(0,+\infty), \\\\
\phi(x,0)=\phi_0(x),\:\:\:\:\:\:\:\:\:\:\:\:\:\:\:\:\:\:\:\:\:\:\:\:\:\:\:\:\:\:\:\:\:\:\:\:\:\:\:\:\:\:\:\:\: \text{in} \: [0,L],\\\\
\phi_t(x,0)= \phi_1(x),\:\:\:\:\:\:\:\:\:\:\:\:\:\:\:\:\:\:\:\:\:\:\:\:\:\:\:\:\:\:\:\:\:\:\:\:\:\:\:\:\:\:\:\:\text{in} \: [0,L],
\end{array}\right.
\end{equation}
with the zero-mean property and initial data $(\phi_0,\phi_1)\in H_{per,m}^3\times H_{per,m}^2$. To be more precise, we have the following result:

\begin{theorem}[Local well-posedness for the Cauchy problem]\label{lwpthm}
Let $(\phi_0,\phi_1)\in H_{per,m}^3\times H_{per,m}^2$. There exists $t_{\rm{max}}>0$ and a unique  local (strong) solution $\phi$ of the Cauchy problem $(\ref{CPKG3})$ satisfying $(\phi,\phi_t)\in C([0,t_{\rm{max}}),H_{per,m}^3\times H_{per,m}^2)\cap C^1([0,t_{\rm{max}}),H_{per,m}^2\times L_{per,m}^2) $. 
\end{theorem}
To prove Theorem $\ref{lwpthm}$, we first need to obtain local strong solutions to the modified problem in $(\ref{CPKG2})$ by applying the abstract semigroup theory as detailed in \cite[Chapter 1, Chapter 6]{pazy}. To this end, we prove that the linear (unbounded) operator
$\displaystyle A=\left(\begin{array}{lll}\partial_x^{-1}&  0\\ 0&  \partial_x\end{array}\right)\left(\begin{array}{lll}0&  1\\ \partial_x^2&  0\end{array}\right)$  defined on $X=H_{per,m}^2\times L_{per,m}^2$ with domain $D(A)=H_{per,m}^3\times H_{per,m}^1$ is a generator of a contraction semigroup $\{S(t)\}_{t\geq0}$ on $X$ (see Lemma $\ref{lemmaC0}$). We also establish the existence of global weak solution associated with the Cauchy problem $(\ref{CPKG})$. This result is essential, since the stability notion adopted here (see Definition \ref{stadef}) requires orbital stability of periodic waves in the energy space  $Y$.

\begin{theorem}[Existence of a weak solution]\label{lwpthm2}	
	Let $L\in (0,2\pi)$ be fixed and consider $(\phi_0,\phi_1)\in Y$. There exists a unique  global (weak) solution $\phi$ of the Cauchy problem $(\ref{CPKG})$ satisfying $(\phi,\phi_t)\in C([0,+\infty),Y)$.
\end{theorem}

\begin{remark}\label{gwpstab}
\indent Some important comments deserve to be highlighted regarding Theorem $\ref{stabthm}$ and the question of the existence of local and global solutions in the energy spaces $H_{per}^1 \times L_{per}^2$, $H_{per,odd}^1 \times L_{per,odd}^2$, and $H_{per,m}^1 \times L_{per,m}^2$. In order to study the stability of periodic waves in the space $H_{per,m}^1 \times L_{per,m}^2$, we must consider the modified Cauchy problem $(\ref{CPKG2})$ instead of the standard Cauchy problem~$(\ref{CPKG})$, which is posed in the spaces $H_{per}^1 \times L_{per}^2$ and $H_{per,odd}^1 \times L_{per,odd}^2$. The existence of local solutions in the last two cases can be obtained by applying the semigroup theory developed in \cite{pazy}, which establishes the existence of local mild solutions in these spaces. In the first case, we can follow \cite{loreno} and \cite{palacios} to conclude the orbital instability of the periodic wave~$(\ref{Sol2})$. In the second space, the mild solution $\phi$ of equation~$(\ref{KF2})$ is odd, and thus satisfies the zero-mean property $\int_0^L \phi(x,t), dx = 0$ for all $t \in [0, t_{\max})$. Therefore, global solutions can be established using the classical Poincar\'e–Wirtinger inequality without further difficulties (see Remark~$\ref{gwp}$ for more details). The orbital stability in $H_{per,odd}^1 \times L_{per,odd}^2$ is then obtained since the restricted linearized operator $\mathcal{L}_{odd}$ in $(\ref{matrixop313})$ is considered in $L_{per,odd}^2 \times L_{per,odd}^2$, with domain $H_{per,odd}^2 \times L_{per,odd}^2$. In fact, using \cite[Proposition 3.8]{loreno}, we establish that the first negative eigenvalue of the linear operator $\mathcal{L}_1=-\omega\partial_x^2-1+3h^2$, defined in the entire space $L_{per}^2$, is associated with an even periodic eigenfunction. Consequently, $n(\mathcal{L}_{1,\text{odd}}) = n(\mathcal{L}_{\text{odd}}) = 0$. In addition, for $c = 0$, we find that $(h', 0)$ is the only element in $\Ker(\mathcal{L})$, and since $h'$ is even, we conclude that $\Ker(\mathcal{L}_{odd}) = {0}$. Therefore, $\mathcal{L}_{odd}$ is a positive linear operator, and the coercivity condition as in $(\ref{positiveL})$
\begin{equation}\label{positiveL123}
(\mathcal{L}_{odd}(p,q), (p,q))_{\mathbb{L}_{per}^2} \geq C||(p,q)||_{\mathbb{L}_{per}^2}^2,
\end{equation}
for all $(p,q) \in H_{per,odd}^2 \times H_{per,odd}^1$, is automatically satisfied, as required for orbital stability. Thus, our result, restricted to the new energy space $H_{per,m}^1 \times L_{per,m}^2$, appears to be more general in the context of the $\phi^4-$equation.
\end{remark}

Our paper is organized as follows. In Section \ref{section3}, we show the well-posedness results for the Cauchy problem $(\ref{CPKG})$ in smooth spaces. The existence of periodic traveling waves for the equation $(\ref{KF2})$ and the spectral analysis for the linearized operators $\mathcal{L}$ and $\mathcal{L}_{\Pi}$ are established in Section \ref{section4}. Finally, the orbital stability of the periodic waves will be shown in Section \ref{section5}.\\

\textbf{Notation.} Here we introduce the basic notation concerning the periodic Sobolev spaces. For a more complete introduction to these spaces we refer the reader to \cite{Iorio}. By $L^2_{per}=L^2_{per}([0,L])$, $L>0$, we denote the space of all square integrable real functions which are $L$-periodic. For $s\geq0$, the Sobolev space
$H^s_{per}=H^s_{per}([0,L])$
is the set of all periodic real functions such that
$\displaystyle
\|f\|^2_{H^s_{per}}= L\displaystyle \sum_{k=-\infty}^{+\infty}(1+|k|^2)^s|\hat{f}(k)|^2 <+\infty,
$
where $\hat{f}$ is the periodic Fourier transform of $f$. The space $H^s_{per}$ is a  Hilbert space with natural inner product denoted by $(\cdot, \cdot)_{H^s_{per}}$. When $s=0$, the space $H^s_{per}$ is isometrically isomorphic to the space  $L^2_{per}$, that is, $L^2_{per}=H^0_{per}$ (see, e.g., \cite{Iorio}). The norm and inner product in $L^2_{per}$ will be denoted by $\|\cdot \|_{L^2_{per}}$ and $(\cdot, \cdot)_{L^2_{per}}$, respectively. 

For $s \geq 0$, we define
$\displaystyle
	H^s_{per,m} = \left\{ f \in H^s_{per}\; ; \; \frac{1}{L}\int_0^L  f(x) \; dx =0 \right\},
$
endowed with  norm and inner product of $H_{per}^s$. Denote the topological dual of $H^s_{per,m}$ by
$
H^{-s}_{per,m}=(H^s_{per,m})' . 
$
In addition, to simplify notation we set
$$\mathbb{H}^s_{per}= H^s_{per} \times H^s_{per},\ \
\mathbb{H}^s_{per,m}= H^s_{per,m} \times H^s_{per,m},\ \
\mathbb{L}^2_{per}= L^2_{per} \times L^2_{per},$$
endowed with their usual norms and scalar products.

\section{Local and global well-posedness}\label{section3}

The aim of this section is to prove Theorems $\ref{lwpthm}$ and $\ref{lwpthm2}$. We begin with the following elementary lemma:

\begin{lemma}\label{lemmaC0}The operator $A=\left(\begin{array}{lll}\partial_x^{-1}&  0\\ 0&  \partial_x\end{array}\right)\left(\begin{array}{lll}0&  1\\ \partial_x^2&  0\end{array}\right)$ defined on $X=H_{per,m}^2\times L_{per,m}^2$ with domain $D(A)=H_{per,m}^3\times H_{per,m}^1$ is a generator of a $C_0-$semigroup of 
contractions $\{S(t)\}_{t\geq0}$ on the space $X$.
\end{lemma}
\begin{proof}
Our goal is to use Lumer-Phillips Theorem (see \cite[Chapter 1, Theorem 4.3]{pazy}). First, we see that $D(A)$ is a dense subspace in $X$ and $A$ is an unbounded 
linear operator defined in $D(A)$. We prove that
$A$ is dissipative. In fact for a given $(\phi,\psi)\in D(A)$, we have that 
\begin{equation*}
 \left(A(\phi,\psi),(\phi,\psi)\right)_{X}=(\partial_x^{-1}\psi,\phi)_{H_{per}^2}  +(\partial_x^3\phi,\psi)_{L_{per}^2}  =(\partial_x\psi,\partial_x^2\phi)_{L_{per}^2}  -(\partial_x^2\phi,\partial_x\psi)_{L_{per}^2} =0.
\end{equation*}
On the other hand, consider $\lambda>0$. We claim that $(\lambda Id -A):D(A)\subset X \rightarrow X$ is onto. Indeed, by considering $(f,g)\in X$, we need to find $(\phi,\psi)\in D(A)$ that solves the equation
\begin{equation}\label{eq1}
	\lambda(\phi,\psi)-A(\phi,\psi)= (f,g).
	\end{equation}
\indent Solving $(\ref{eq1})$ is equivalent to finding a solution to the system
\begin{equation}\label{eq2}
	\left\{\begin{array}{llll}\lambda \phi-\partial_x^{-1}\psi=f,\\\\
		\lambda \psi-\partial_x^3\phi=g.
	\end{array}\right.
	\end{equation}
Hence, solving the system (\ref{eq2}) is equivalent to finding a solution of the single equation 
\begin{equation}\label{eq3}
	\lambda^2\phi-\partial_x^2\phi=\partial_x^{-1}g+\lambda f.
\end{equation}
However, equation $(\ref{eq3})$ can be solved by a standard application of the Lax-Milgram Lemma. Then, using Lumer-Philips Theorem, we obtain that $A$ is a generator of a $C_0-$semigroup of contractions $\{S(t)\}_{t\geq0}$ on $X$. This completes the proof of the lemma.
\end{proof}

\indent The next lemma establishes the existence of a local solution of the Cauchy problem $(\ref{CPKG2})$.

\begin{lemma}\label{lwpthm1}
Let $(\phi_0,\psi_0)\in H_{per,m}^3\times H_{per,m}^1$. There exists $t_{\rm{max}}>0$ and a unique local (strong) solution   $(\phi,\psi)$ of the Cauchy problem $(\ref{CPKG2})$ satisfying $(\phi,\psi)\in C([0,t_{\rm{max}}),H_{per,m}^3\times H_{per,m}^1)\cap C^1([0,t_{\rm{max}}),H_{per,m}^2\times L_{per,m}^2) $. 
\end{lemma}

\begin{proof}
	
By Lemma $\ref{lemmaC0}$, we have that $A$ is a generator of a $C_0-$semigroup of 
contractions $\{S(t)\}_{t\geq0}$ on $X$. Let us consider $(\phi_0,\psi_0)\in H_{per,m}^3\times H_{per,m}^1=D(A)\subset X$. Initially, we show the existence of $t_{\rm{max}}>0$ and a unique function 
\begin{equation}\label{eq1.lemma}
	U=(\phi,\psi)\in C([0,t_{\rm{max}}),H_{per,m}^2([0,L])\times L_{per,m}^2([0,L])), 
\end{equation}
such that, for $t\in [0,t_{\rm{max}})$, we have that $U(t)$ solves the integral equation
\begin{equation}\label{eq2.lemma}
	U(t)=(\phi(\cdot, t), \psi(\cdot, t))=S(t)(\phi_0,\psi_0)+\int_{0}^{t}S(t-s)(0,\partial_x(\phi(\cdot,s)-\phi^3(\cdot,s)))ds. 
\end{equation}
First, let us define the function $Q:H_{per,m}^2\times L_{per,m}^2\rightarrow H_{per,m}^2\times L_{per,m}^2$ given by $$Q(\phi,\psi)=(0, \partial_x(\phi-\phi^3)).$$ We need show that $Q$ is well-defined. To do so, it suffices to prove that $\partial_x \phi ^3 \in L_{per,m}^2 $ for all $\phi \in H_{per,m}^2$. Indeed, the zero-mean property is clearly satisfied by the periodicity of $\phi$. We also have that if $\phi \in H_{per,m}^2$, then $|\partial_x\phi|^2 \in L_{per,m}^1$. Furthermore, using the embedding $H_{per}^1 \hookrightarrow  L_{per}^p$ for $p\in [1,+\infty]$, we note that $|\phi|^4\in  L_{per,m}^\infty$. It then follows from Hölder's inequality that
\begin{align*}
\int_{0}^{L} |\partial_x \phi ^3 (x)|^2dx=9 \int_{0}^{L} |\phi(x)|^4|\partial_x \phi(x) |^2dx\leq 9 \||\phi|^4\|_{ L_{per}^\infty}\||\partial_x \phi|^2\|_{L_{per}^1}<+\infty,
\end{align*}consequently, $Q$ is well-defined. Next, we prove an important property: let $R>0$ be fixed and suppose that 
 $(\phi_1, \psi_1), \  (\phi_2, \psi_2)\in X=H_{per,m}^2\times L_{per,m}^2$ satisfy $\| (\phi_1, \psi_1)\|_{X}\leq R$ and  $\| (\phi_2, \psi_2)\|_{X}\leq R$. There exists $M=M(L,R)>0$ such that
   \begin{equation}\label{eq7.1.lemma}  
 	\|Q(\phi_1, \psi_1)-Q(\phi_2, \psi_2)\|_X\leq  M \|(\phi_1,\psi_1)-(\phi_2, \psi_2)\|_X.
 \end{equation}
Indeed, we first note that
 \begin{align}\label{eq3.lemma}
 	\|Q(\phi_1, \psi_1)-Q(\phi_2, \psi_2)\|_X\leq\| \phi_1-\phi_2\|_{H_{per}^2}+ \|\partial_x (\phi_1^3-\phi_2^3)\|_{L_{per}^2} .
 \end{align}
A straightforward calculation shows that the second term on the right-hand side of $(\ref{eq3.lemma})$ can be expressed as
 \begin{equation}
 	\|\partial_x (\phi_1^3-\phi_2^3)\|_{L_{per}^2}\leq 3\|\phi_1\|_{L_{per}^\infty}^2 \|\partial_x\phi_1-\partial_x\phi_2 \|_{L_{per}^2}+3\|\partial_x\phi_2\|_{L_{per}^\infty} \| \phi_1 +\phi_2  \|_{L_{per}^\infty} \|\phi_1 -\phi_2  \|_{L_{per}^2}. 
\label{eq5.lemma}
\end{equation}

Using the Sobolev embeddings $H_{per,m}^2 \hookrightarrow H_{per,m}^1\hookrightarrow L_{per,m}^2$ in \eqref{eq5.lemma}, together with the bound by 
$R$, we obtain the existence of a constant $M_1=M_1(L,R)>0$ such that 
 \begin{equation}\label{eq6.lemma}
 		\|\partial_x (\phi_1^3-\phi_2^3)\|_{L_{per}^2}\leq M_1 \|\phi_1 -\phi_2  \|_{H_{per}^2}. 
 		\end{equation}
 		By \eqref{eq3.lemma} and \eqref{eq6.lemma}, there exist a constant $M=M(L,R)>0$ satisfying
 \begin{equation}\label{eq7.lemma}  
 		\|Q(\phi_1, \psi_1)-Q(\phi_2, \psi_2)\|_X\leq  M \|(\phi_1,\psi_1)-(\phi_2, \psi_2)\|_X.
\end{equation} This fact establishes the desired result.

For a given $T>0$, let us define the set
\begin{equation}\label{eq8.lemma}  
	\Upsilon=\left\{(\phi,\psi)\in C([0,T], X);\ \sup_{t\in [0,T]}\|(\phi(\cdot,t), \psi(\cdot,t))\|_X\leq 1+\|(\phi_0, \psi_0)\|_X\right\}.
\end{equation}
The set $	\Upsilon$ is a complete metric space because it is closed in $C([0,T],X)$ with the supremum norm. Next, let us consider the mapping $\Psi: \Upsilon\rightarrow\Upsilon$ defined, for each $t\in [0,T]$, by
\begin{equation}\label{eq9.lemma}  
	\Psi(\phi(\cdot,t), \psi(\cdot, t))=S(t)(\phi_0,\psi_0)+\int_{0}^{t}S(t-s)Q(\phi(\cdot,s),\psi(\cdot,s))ds.
\end{equation}
In order to use Banach’s Fixed Point Theorem to prove the existence and uniqueness of the abstract Cauchy problem \eqref{CPKG2}, we show that the function $\Psi$ is well defined on an open ball with radius $r>0$ and that it is a strict contraction. To do so, let us consider $r=1+\|(\phi_0, \psi_0)\|_X>0$ and an arbitrary and fixed 
$(\phi,\psi)\in \Upsilon$. We need to choose $T>0$ to ensure the well-definedness of $\Psi$. In fact, by \eqref{eq7.lemma} we obtain that, for all $t\in [0,T]$, there exists a constant $M_2=M_2(L,r)>0$ such that
\begin{equation}\label{eq10.lemma}  
		\|Q(\phi(\cdot,t), \psi(\cdot,t))-Q(\phi_0, \psi_0)\|_X\leq  M_2 \|(\phi(\cdot,t),\psi(\cdot,t))-(\phi_0, \psi_0)\|_X. 
\end{equation}
By equation \eqref{eq10.lemma} and the fact that $S(t)$ is a $C_0$-semigroup of contractions, we obtain 
\begin{equation}\label{eq11.lemma}  
\|	\Psi(\phi(\cdot,t), \psi(\cdot, t))\|_X\leq \|(\phi_0,\psi_0)\|_X+ M_2T[1+2\|(\phi_0,\psi_0)\|_X]  +T\|(0,\partial_x(\phi_0-\phi^3_0))\|_X. 
\end{equation}
Considering 
\begin{equation}\label{eq13.lemma}  
	0<T^*=\left\{  M_2[1+2\|(\phi_0,\psi_0)\|_X]+\|(0,\partial_x(\phi_0-\phi^3_0))\|_X \right\}^{-1}<+\infty, 
\end{equation}
and by redefining $T$ so that $0<T\leq T^*$, it follows from \eqref{eq11.lemma} and \eqref{eq13.lemma} that for all $t\in [0,T]$ we have
\begin{equation}\label{eq14.lemma}  
	\|	\Psi(\phi(\cdot,t), \psi(\cdot, t))\|_X\leq1+\|(\phi_0,\psi_0)\|_X=r, 
\end{equation}proving the well-definedness of the function $\Psi$.\\
\indent  Next, we prove that $\Psi$ is a strict contraction. It is important to mention that, from now on, if necessary, we redefine $T$ to prove that $\Psi$ is a contraction. To this end, consider $(\phi_1,\psi_1), (\phi_2,\psi_2)\in \Upsilon$. Using a similar argument as in \eqref{eq7.lemma}, we deduce
\begin{equation}\label{eq16.lemma}  
		\|Q(\phi_1, \psi_1)-Q(\phi_2, \psi_2)\|_{C([0,T],X)}\leq  M_2T\|(\phi_1,\psi_1)-(\phi_2, \psi_2)\|_{C([0,T],X)}.
\end{equation} Since $T\leq T^{*}$ and, by $(\ref{eq13.lemma})$, it follows that $T^{*}<\frac{1}{M_2}$, we conclude that $\Psi$ is a strict contraction. Under these conditions, Banach’s Fixed Point Theorem guarantees the existence of a unique function $(\phi, \psi)\in \Upsilon$ such that $\Psi( \phi, \psi)=(\phi, \psi),$ that is, \eqref{eq2.lemma} holds. \\
\indent In what follows, let us consider $t_{\rm{max}}=T^{*}$. Using Gronwall’s inequality, together with the fact that $Q$ satisfies $(\ref{eq7.1.lemma})$, it is possible to verify that the function $(\phi, \psi)\in C([0,t_{\rm{max}}), X)$ above is the unique mild solution of the Cauchy problem on $[0, t_{\rm{max}})$.\\
\indent The next step is to prove that $(\phi,\psi)\in C([0,t_{\rm{max}}),D(A))\cap C^1([0,t_{\rm{max}}),X) $. For that, consider $U_0=(\phi_0, \psi_0)$, $U(t)= (\phi(\cdot, t), \psi(\cdot, t))$ and $ v(t)=\int_{0}^{t}S(t-s)Q(U(s))ds$. Since
\begin{equation}\label{eq17.lemma}  
	U(t)= S(t) U_0+v(t),
\end{equation}we obtain by
\cite[Chapter 1, Theorem 2.4]{pazy} that $v(t)\in D(A)$, $S(t)U_0 \in D(A)$ and 
\begin{equation}\label{eq18.lemma}  
	\frac{d}{dt}S(t) U_0=AS(t)U_0=S(t)AU_0.
\end{equation}
Furthermore, using the property in \eqref{eq7.1.lemma} for the function $Q$, the conditions that $S(t)$ is of class $C_0$, and $ U\in C([0,t_{\rm{max}}),X)$, we prove that $v$ it is differentiable and satisfies
\begin{equation}\label{eq19.lemma}  
		\frac{d}{dt}v(t)=A(v(t))+Q(U(t)).
\end{equation}
From \eqref{eq18.lemma} and \eqref{eq19.lemma} we deduce that \begin{equation}\label{eq20.lemma}  
 	\frac{d}{dt}U(t)= AU(t)+Q(U(t)) ,
\end{equation} and hence, $ U\in C([0,t_{\rm{max}}),D(A))\cap C^1([0,t_{\rm{max}}),X)$ is the unique local (strong) solution of the Cauchy problem $(\ref{CPKG2})$. 
\end{proof}

\noindent \textit{Proof of Theorem $\ref{lwpthm}$.} Let $(\phi_0,\phi_1)\in H_{per,m}^3\times H_{per,m}^2$. Hence $(\phi_0,\psi_0)=(\phi_0,\partial_x \phi_1)\in D(A)$ and by Lemma $\ref{lwpthm1}$ there exists $t_{\rm{max}}>0$ and a unique strong solution\noindent $$(\phi,\psi)\in C([0,t_{\rm{max}}),H_{per,m}^3\times H_{per,m}^1)\cap C^1([0,t_{\rm{max}}),H_{per,m}^2\times L_{per,m}^2)$$ of the Cauchy problem 
$(\ref{CPKG2})$. Given that $(\phi,\psi)$ is a strong solution of equation $(\ref{CPKG2})$, it follows that the pair $(\phi,\psi)$ satisfies the following system of partial differential equations
\begin{equation}\label{sysKG}
\left\{\begin{array}{llll}\phi_t=\partial_x^{-1}\psi,\\\\
\psi_t=\partial_x(\phi_{xx}+\phi-\phi^3).
\end{array}\right.
\end{equation}
\indent By differentiating the first equation in $(\ref{sysKG})$ with respect to $t$, applying the operator $\partial_x^{-1}$ to the second equation, and comparing the results, we find that $\phi$ satisfies the equation 
\begin{equation}\label{zmphi4}
\phi_{tt}-\phi_{xx}-\phi+\phi^3-\frac{1}{L}\int_0^L\phi^3dx=0.
\end{equation}
Also, from $\phi_t=\partial_x^{-1}\psi$, we deduce that  $(\phi_0,\psi_0)=(\phi(0),\psi(0))=(\phi_0,\partial_x\phi_t(0))=(\phi_0,\partial_x\phi_1)$. This demonstrates that $\phi$ is the unique strong solution to the Cauchy problem $(\ref{CPKG3})$ that satisfies $(\phi_0,\partial_x^{-1}\psi_0)=(\phi_0,\phi_1)=(\phi(0),\phi_t(0))$ and $$(\phi,\phi_t)\in C([0,t_{\rm{max}}),H_{per,m}^3\times H_{per,m}^2)\cap C^1([0,t_{\rm{max}}),H_{per,m}^2\times L_{per,m}^2).$$  This concludes the proof of the theorem. 
\begin{flushright}
$\blacksquare$
\end{flushright}
 
\begin{remark}\label{cqEF} We can deduce the two basic conserved quantities in $(\ref{E})$ and $(\ref{F})$ associated with the evolution equation $(\ref{KF2})$. In fact, since $(\phi, \phi_t)\in C^1([0,t_{\rm{max}}),H_{per,m}^2\times L_{per,m}^2)$, we obtain by Lemma \ref{lwpthm1}, after multiplying $(\ref{zmphi4})$ by $\phi_t$ and integrating the final result over $[0,L]$, that
\begin{equation}\label{ded2}
\frac{1}{2}\frac{d}{dt}\int_0^L\left(  \phi_x^2+\phi_t^2-\phi ^2 +\frac{1}{2}\phi^4\right)-\frac{1}{L}\int_0^L\phi^3dx\int_{0}^{L}\phi_tdx=0. 
\end{equation}
 Then, by \eqref{ded2} we have the conserved quantity in $(\ref{E})$. \\
\indent We prove that $\mathcal{F}(\phi,\phi_t)=\displaystyle\int_{0}^{L}\phi_x\phi_t dx$ is also a conserved quantity. Indeed, by \eqref{zmphi4} we have that 
\begin{equation}\label{ded3}\begin{array}{llll}
	\displaystyle\frac{d}{dt}\mathcal{F}(\phi,\phi_t)&=& \displaystyle\int_{0}^{L}(\phi_t \phi_{tx}+ \phi_x \phi_{tt}) dx \\\\ &= &\displaystyle\int_{0}^{L}\left(\phi_t \phi_{tx}+ \phi_x \phi_{xx}+\phi_x \phi-\phi_x \phi^3+\phi_x\frac{1}{L}\int_0^L\phi^3dx\right) dx \\\\ &= &\displaystyle\int_{0}^{L}\left(\frac{1}{2}\frac{d}{dx} \phi_t ^2
	+ \frac{1}{2}\frac{d}{dx} \phi_x ^2+\frac{1}{2}\frac{d}{dx} \phi ^2-\frac{1}{4}\frac{d}{dx} \phi ^4+\phi_x\frac{1}{L}\int_0^L\phi^3dx\right)dx.
\end{array}\end{equation}
Due to the periodicity of the functions $\phi$, $\phi_x$, and $\phi_t$, we obtain by $(\ref{ded3})$ that $\mathcal{F}$ is conserved quantity.
\end{remark}
 
 \begin{remark}\label{gwp}
  We prove that the local solution 
 $(\phi,\phi_t)$ in Theorem $\ref{lwpthm}$ extends globally in $Y$. By Remark \ref{cqEF}, and using Young’s inequality, we obtain that
 \begin{equation}\label{est2}\begin{array}{lllll}
 		\displaystyle\int_0^L[\phi_x(x,t)^2+\phi_t(x,t)^2]dx&=&2\mathcal{E}(\phi(t),\phi_t(t))+\displaystyle\int_0^L\left[\phi (x,t)^2 -\frac{1}{2}\phi(x,t)^4\right]dx\\\\
 		&=&2\mathcal{E}(\phi_0,\phi_1)+\displaystyle\int_0^L\left[\phi (x,t)^2 -\frac{1}{2}\phi(x,t)^4\right]dx \\\\
 		 &\leq &2\displaystyle\mathcal{E}(\phi_0,\phi_1)+\frac{L}{2}.
 	\end{array}
 \end{equation}
 This implies that $(\phi,\phi_t)\in L^{\infty}([0,+\infty),Y)$. In other words, the strong solution $(\phi,\phi_t)$ of the Cauchy problem $(\ref{CPKG3})$ is global in time in $H_{per,m}^1\times L_{per,m}^2$. 
  \end{remark}

\noindent \textit{Proof of Theorem $\ref{lwpthm2}$.} Let $(\phi_0,\phi_1)\in H_{per,m}^1\times L_{per,m}^2$. By density, there exists a sequence $\{ (\phi_{0,n},\phi_{1,n})\}_{n\in\mathbb{N}} \subset H_{per,m}^3\times H_{per,m}^2$ such that
$(\phi_{0,n},\phi_{1,n})$ converges to $ (\phi_0,\phi_1)  $ in $H_{per,m}^1\times L_{per,m}^2$. For the regular initial data $(\phi_{0,n},\phi_{1,n})$, we have by Theorem $\ref{lwpthm}$ the corresponding sequence of solutions 
\begin{align} \label{teo.fraco.eq1}
	(\phi_n, \phi_{t,n})\in  C([0,t_{\rm{max}}),H_{per,m}^3\times H_{per,m}^2)\cap C^1([0,t_{\rm{max}}),H_{per,m}^2\times L_{per,m}^2),
\end{align}so that 
\begin{align} \label{teo.fraco.eq2}
	\phi_{tt,n}-\phi_{xx,n}-\phi_n+\phi_n^3-\frac{1}{L}\int_0^L\phi_n^3dx=0, \  {\rm{in}} \: [0,L]\times(0,t_{\rm{max}}).
\end{align}
\indent In addition, by Remark $\ref{gwp}$, we see that the pair $(\phi_n,\phi_{t,n})$ is bounded and global in time in the space $Y$. To simplify the notation, let us denote $u=\phi_n$, $v=\phi_r$, $u_t=\phi_{t,n}$ and $v_t=\phi_{t,r}$. Under these conditions, defining $w=u-v$, we have that
\begin{align} \label{teo.fraco.eq4}
	w_{tt}-w_{xx}-w+w(u^2+uv+v^2)-\frac{1}{L}\int_0^Lw(u^2+uv+v^2)dx=0.
\end{align}
Multiplying $\eqref{teo.fraco.eq4}$ by $w_t$ and integrating in $x$ over the interval $[0,L]$, we see that the last term in the calculation cancels because $w_t$ has zero mean. Thus, it follows that
\begin{align} \label{teo.fraco.eq5}
\frac{1}{2}\frac{d}{dt}	\displaystyle\int_0^L(w_t^2+w_x^2)dx\leq  \frac{1}{2}\frac{d}{dt}\int_0^Lw^2dx +\int_0^L|ww_t||u^2+uv+v^2|dx.
\end{align}
\indent Moreover, using Young's inequality in \eqref{teo.fraco.eq5} and integrating the result over the interval $[0,t]\subset[0,t_{\rm{max}})$, we obtain
\begin{equation} \label{teo.fraco.eq6}\begin{array}{lllll}
\displaystyle	\frac{1}{2} 	\displaystyle\int_0^L(w_t^2+w_x^2)dx&\leq &\displaystyle \frac{1}{2}\int_0^Lw^2dx  +	\frac{1}{2} 	\displaystyle\int_0^L(w_{t,0}^2+w_{x,0}^2)dx \\\\
	&+&\displaystyle	\frac{1}{2}\int_0^t\int_0^Lw^2|u^2+uv+v^2|dxdt\\\\
	&+&\displaystyle	\frac{1}{2}\int_0^t\int_0^Lw_t^2|u^2+uv+v^2|dxdt. 
\end{array}
\end{equation} By Remark $\ref{gwp}$ and using the embedding $H_{per,m}^1\hookrightarrow L_{per,m}^{\infty}$, we guarantee the existence of a constant $M_3>0$ depending on $L$ such that $ \|u^2+uv+v^2\|\displaystyle_{L_{per,m}^\infty}\leq M_3$. Using Fourier series together with Parseval's identity, we can also prove the Poincaré–Wirtinger inequality in $H_{per,m}^1$ as follows:
\begin{align}  \label{PWineq}
	\frac{1}{2}\int_0^Lw^2dx\leq 	\frac{1}{2}\biggl(\frac{L}{2\pi}\biggl)^2\int_0^Lw_x^2dx.
\end{align} By \eqref{teo.fraco.eq6} and \eqref{PWineq}, it follows that
\begin{align} \label{teo.fraco.eq8}
	\begin{array}{lllll}
		\displaystyle	\frac{1}{2} \int_0^L(w_t^2+w_x^2)dx&\leq &\displaystyle \frac{1}{2}\biggl(\frac{L}{2\pi}\biggl)^2\int_0^Lw_x^2dx +	\frac{1}{2} 	\displaystyle\int_0^L(w_{t,0}^2+w_{x,0}^2)dx \\\\
		&+&\displaystyle	\frac{M_3}{2}\biggl(\frac{L}{2\pi}\biggl)^2 \int_0^t \int_0^Lw_x^2dxdt+	\frac{M_3}{2}\int_0^t\int_0^Lw_t^2 dxdt. 
	\end{array}
\end{align}For $L\in (0,2\pi)$, let us consider $M_4=\displaystyle \biggl[ 1-\biggl(\frac{L}{2\pi}\biggl)^2\biggl]<1$ 
 in \eqref{teo.fraco.eq8}. Thus, 
\begin{align} \label{teo.fraco.eq9}	\begin{array}{lllll}
	 	\displaystyle\int_0^L(w_t^2+w_x^2)dx\leq\displaystyle  	\frac{1}{M_4} 	\displaystyle\int_0^L(w_{t,0}^2+w_{x,0}^2)dx +\displaystyle	\frac{M_3}{M_4} \int_0^t \int_0^L(w_t^2+w_x^2)dxdt . 
	\end{array}
\end{align}Applying Gronwall's inequality to \eqref{teo.fraco.eq9}, we conclude
\begin{align} \label{teo.fraco.eq10}\displaystyle\int_0^L(w_t^2+w_x^2)dx\leq\frac{1}{M_4} 	\displaystyle\biggl[\int_0^L(w_{t,0}^2+w_{x,0}^2)dx\biggl] e^{\frac{M_3}{M_4}T},
\end{align} where $T>0$ is arbitrary, but fixed. Since $w=u-v$, \eqref{teo.fraco.eq10} shows that $	(\phi_n, \phi_{t,n})$ is a Cauchy sequence in $L^\infty([0,T],Y)$. 
Hence, there exists $(\phi, \phi_t) \in L^\infty([0,T],Y)$ such that 
\begin{equation}  \label{teo.fraco.eq11}
	(\phi_n, \phi_{t,n})\rightarrow (\phi, \phi_t) \quad {\rm{in}} \quad L^\infty([0,T],Y).  
\end{equation}
\indent Defining $U_0=(\phi_0,\phi_1)$, $U=(\phi, \phi_{t})$, $U_{n}=(\phi_n, \phi_{t,n})$, $U_{0,n}=(\phi_{0,n}, \phi_{1,n})$, and using the conserved quantity $\mathcal{E}$ in $(\ref{E})$, together with the convergence in \eqref{teo.fraco.eq11}, we obtain
\begin{align}\label{teo.fraco.eq12}
\begin{array}{ccc}
	\mathcal{E}(U_n(t)) & = & \mathcal{E}(U_{0,n}) \\
	\mathrel{\downarrow} & & \mathrel{\downarrow} \\
	\mathcal{E}(U(t)) & = & \mathcal{E}(U_0)
\end{array}
\end{align}
By applying the convergences in \eqref{teo.fraco.eq12} and the arguments in \cite[Lemma 2.4.4]{cazenave}, we establish that $	(\phi, \phi_{t})\in  C([0,T],Y)$ for all $T>0$. Furthermore, by standard arguments of passage to the limit, one can show that $(\phi, \phi_t)$ is a weak solution of $\eqref{CPKG}$ in $C([0,T],Y)$ with initial data $(\phi_0,\phi_1)\in Y$ provided that $L\in (0,2\pi)$.\\
\indent Inequality $(\ref{teo.fraco.eq10})$ and the fact that $w = u - v$ also imply that the weak solution is unique in $C([0,T], Y)$. In addition, returning to the estimate $(\ref{est2})$, but now using the weak solution $(\phi, \phi_t) \in C([0,T], Y)$ instead of strong solution, we conclude that in fact $T = +\infty$, so that $(\phi, \phi_t) \in C([0,+\infty), Y)$, as stated in Theorem $\ref{lwpthm2}$.
\begin{flushright}
	$\blacksquare$
\end{flushright}

\begin{remark}\label{remweaksol} Just to make clear that our notion of global weak solution mentioned in Theorem $\ref{lwpthm2}$ reads as follows: we say that $\phi$ is a global weak solution for the problem $(\ref{CPKG})$ if for all $\mathsf{p}\in H_{per,m}^1$, we have 
	$$
	\langle \phi_{tt}(\cdot,t), \mathsf{p}\rangle_{H^{-1}_{per,m},H_{per,m}^1} + \int_{0}^{L} \phi_x(x,t)\mathsf{p}_x(x) dx-\int_{0}^{L} \phi(x,t)\mathsf{p}(x)dx + \int_{0}^{L} \phi(x,t)^3\mathsf{p}(x)dx = 0,
	$$
	a.e. $t \in [0,+\infty)$. It is important to mention that, since $\mathsf{p}\in H_{per,m}^1$, we have $$\frac{1}{L}\int_0^L\phi(x,t)^3dx\int_0^L\mathsf{p}(x)dx=0,$$ and therefore the term $\displaystyle\int_0^L \phi(x,t)^3dx$ that appears in \eqref{zmphi4} does not intervene in the definition of the weak solution, since its inner product with a zero-mean function in $L_{per}^2$ is zero.
	\end{remark}

\section{Existence of periodic waves and spectral analysis}\label{section4}

\subsection{Existence of periodic waves}

Substituting the traveling wave
solution of the form $\phi(x,t)=h_c(x-ct)$ into $(\ref{KF2})$, one has
\begin{equation}
c^2h''-h''-h+h^3=0.
\label{travSG}\end{equation} Since $\omega=\omega(c)=1-c^2>0$ for $c\in (-1,1)$, we
obtain by $(\ref{travSG})$ the following second order ordinary differential equation
\begin{equation}\label{travSG1}
-\omega h''-h+h^3=0.
\end{equation}
\indent Consider the ansatz of snoidal type
\begin{equation}
h(x)=a\mbox{sn}\left(bx;k\right).
	\label{travSG2}\end{equation} 
Here, $k\in (0,1)$ is the modulus of the elliptic function and the constants $a,b\in \mathbb{R}$ necessitate to be determined. We need to use some useful properties associated with the Jacobi elliptic functions (for details see \cite{byrd}). Indeed, differentiating \eqref{travSG2}, we see that
\begin{align} 	\label{travSG4}
	h''(x)=&-ab^2\mbox{sn}\left(bx;k\right)[\mbox{dn}^2\left(bx;k\right)+k^2\mbox{cn}^2\left(bx;k\right)] . \end{align} 
 Using the identities $\mbox{dn}^2=1-k^2\mbox{sn}^2$ and $\mbox{cn}^2=1-\mbox{sn}^2$ in \eqref{travSG4}, it follows that
 \begin{align} \label{travSG5}
 	h''(x)=&-ab^2\mbox{sn}\left(bx;k\right)[1+k^2-2k^2\mbox{sn}^2\left(bx;k\right)] . \end{align} 
 \indent On the other hand, substituting \eqref{travSG2} and \eqref{travSG5} into \eqref{travSG1}, we have that
  \begin{align} \label{travSG6}
a [b^2\omega(1+k^2)-1]\mbox{sn}\left(bx;k\right)+a[a^2-2b^2k^2\omega ] \mbox{sn}^3\left(bx;k\right)=0. \end{align} 
Consider $a\neq 0$. By \eqref{travSG6}, we can suppose that $a$ and $b$ satisfy
 \begin{align} \label{travSG7}
 b^2 =\frac{1}{\omega(1+k^2)} \quad  {\rm and} \quad  a^2=2b^2k^2\omega . \end{align}Taking into account the positive roots, we obtain an explicit solution $h$ for \eqref{travSG1} given in terms of the Jacobi elliptic functions as
\begin{equation}
	h_\omega(x)=\frac{\sqrt{2}k}{\sqrt{k^2+1} }\mbox{sn}\left(\frac{1}{\sqrt{\omega(1+k^2)}}\; x;k\right).
	\label{travSG8}\end{equation}

In addition, since the Jacobi elliptic function of snoidal kind is periodic with real period equal to $4K(k)$,
we automatically have
\begin{equation}\label{speed1}
\frac{1}{\sqrt{\omega}}=\frac{4K(k)\sqrt{1+k^2}}{L}.
\end{equation}
\indent At this point, we must ensure that, for fixed $L\in(0,2\pi)$, the condition $\omega\in (0,1)$. Since the complete elliptic integral of the first kind is given by $K(k)=\displaystyle\int_{0}^{\frac{\pi}{2}}\frac{1}{\sqrt{1-k^2\sin^2(\theta)}}d\theta,$ we know that $K(k)>\frac{\pi}{2}$. Combining this with the assumption on $k\in (0,1)$, it follows from the relation in \eqref{speed1} that
\begin{equation}\label{speed2}
0<\omega=\frac{L^2}{16K^2(k)(1+k^2)}<\frac{4\pi^2}{4\pi^2 }=1.
\end{equation}
Furthermore, from $(\ref{speed1})$ we deduce that
$\frac{dk}{d\omega}>0$ and by implicit function theorem, we
get
$\omega\in\left(0,\frac{L^2}{4\pi^2}\right)\mapsto h_\omega\in
H_{per,m}^\infty([0,L])$ is smooth.

We can enunciate the following
result.

\begin{proposition}\label{prop.WS}
Let $L\in(0,2\pi)$ be fixed. There exists a smooth curve of
periodic traveling wave solutions for the equation $(\ref{travSG1})$
with $\omega=1-c^2$, $\omega>0$, given by

\begin{equation}
\displaystyle
c\in\left(-\sqrt{1-\frac{L^2}{4\pi^2}},\sqrt{1-\frac{L^2}{4\pi^2}}\right)
\mapsto h_{c}=h_{\omega(c)}\in H_{per,m}^\infty([0,L]),\label{curveeqSG1}\end{equation}
where
\begin{equation} 
	 h_{c}(x)=\frac{\sqrt{2}k}{\sqrt{k^2+1} }{\rm{sn}}\left(\frac{1}{\sqrt{(1-c^2)(1+k^2)}}\;x;k\right).\label{existSG1}
\end{equation}  
\end{proposition}

\begin{flushright}
 ${\blacksquare}$
\end{flushright}

\begin{remark}
 A \textit{kink wave} solution of equation $(\ref{travSG})$ can be obtained by analyzing certain asymptotic properties of the snoidal Jacobi elliptic function given in $(\ref{existSG1})$.  Indeed, let $0<|c|<1$ be fixed. Since
$\mbox{sn}(\cdot,1^{-})\approx \tanh(\cdot)$, then the function
\begin{equation}
h_{c}(y)=\tanh\left(\frac{y}{\sqrt{2(1-c^2)}}\right), \label{kink1} \end{equation}
is an explicit kink wave solution for the equation $(\ref{travSG})$. The anti-kink solution is then given by $h_{c}(y)=\displaystyle-\tanh\left(\frac{y}{\sqrt{2(1-c^2)}}\right)$.
\end{remark}

\subsection{Spectral analysis for the $\phi^4-$equation}\label{spectralKG} 

 Let $L\in (0,2\pi)$ be fixed and consider $c \in (-1,1)$. The main objective of this section is to study the non-positive spectrum of the  operator $\mathcal{L}:H^2_{per}([0,L])\times H^1_{per}([0,L]) \subset\mathbb{L}^2_{per}([0,L])\rightarrow \mathbb{L}^2_{per}([0,L])$ defined in \eqref{matrixop313}. \\
\indent First, by $(\ref{travSG})$, we see that $h' \in \Ker(\mathcal{L}_1)$, where $\mathcal{L}_1$ is the linear operator defined as
\begin{equation}\label{opL1}
\mathcal{L}_1=-\omega\partial_x^2-1+3h^2.
\end{equation}
In addition, since $h$ is odd, then $h'$ has exactly two zeros in the half-open interval $[0,L )$. This implies from \cite[Theorem 3.1.2]{eas} that zero is the second or the third eigenvalue of  $\mathcal{L}_1$ (see also \cite{magnus} for additional results). The next result gives us, in fact, that zero is also the second eigenvalue of the operator $\mathcal{L}$ which is simple.

\begin{lemma}\label{2eigenvalues}  Let $L\in (0,2\pi)$ be fixed.
The operator $\mathcal{L}$ in \eqref{matrixop313} has exactly one negative eigenvalue which is simple. Zero is a simple eigenvalue with associated eigenfunction $(h',ch'')$. In addition, the rest of the
spectrum is constituted by a discrete set of eigenvalues.
\end{lemma}
\begin{proof}
See \cite[Proposition 3.8]{loreno}.
\end{proof}
 
 \indent Before studying the spectral information concerning operator $\mathcal{L}_{\Pi}$ in $(\ref{opconstrained2})$, we need to establish some basic facts. In fact, consider the constrained space $S_1= [1] \subset \Ker(\mathcal{L}_1)^{\perp}=[h']^{\bot}$, which is associated with the auxiliary linear operator $$\mathcal{L}_{1\Pi}=\mathcal{L}_1-\frac{3}{L}(h^2,\cdot)_{L_{per}^2}.$$ Let us define the number
 $
 D_1=(\mathcal{L}_1^{-1} 1,1)_{L^2_{per}}.
 $
 We can use the Index Theorem for self-adjoint operators in \cite[Theorem 5.3.2]{kapitula} and \cite[Theorem 4.1]{pel-book}, to obtain the exact quantity of negative eigenvalues and the dimension of the kernel of $\mathcal{L}_{1\Pi}$. Indeed, since $\Ker(\mathcal{L}_1)=[h']$, one has 
 \begin{equation}\label{indexformula12}
 	\text{n}(\mathcal{L}_{{1\Pi}})=\text{n}(\mathcal{L}_1)-{\rm n}_0-{\rm z}_0
 \end{equation}
 and
 \begin{equation}\label{indexformula123}
 	\text{z}(\mathcal{L}_{{1\Pi}})=\text{z}(\mathcal{L}_1)+{\rm z}_0,
 \end{equation}
 where $\text{n}(\mathcal{A})$ and $\text{z}(\mathcal{A})$ denote the number of negative eigenvalues and the dimension of the kernel of a certain linear operator $\mathcal{A}$ (counting multiplicities). In addition, the numbers ${\rm n}_0$ and ${\rm z}_0$ are defined respectively as
 \begin{equation}\label{n0z0}
 	{\rm n}_0=
 	\begin{cases}
 		1, \: \text{if} \: D_1<0, \\
 		0, \: \text{if} \: D_1 \geq 0,\ \
 	\end{cases}
 	\quad \text{and} \quad\ 
 	{\rm z}_0=
 	\begin{cases}
 		1, \: \text{if} \: D_1=0, \\
 		0, \: \text{if} \: D_1 \neq 0.
 	\end{cases}
 \end{equation}
 
\indent The following result provides the precise spectral information of the operator $\mathcal{L}_{\Pi}$.

\begin{prop}\label{leman1}
 Let $L\in (0,2\pi)$ be fixed. The linear operator $\mathcal{L}_{\Pi}$ in $(\ref{opconstrained2})$ has no negative eigenvalues and $(h',ch'')$ is a simple eigenfunction associated with the zero eigenvalue. 
\end{prop}
\begin{proof}

In order to count the negative eigenvalues, it is necessary to note that $\mathcal{L}_{\Pi}$ is the constrained operator $\mathcal{L}$ defined in $\mathbb{L}_{per,m}^2$ with constrained space $$S= [(1,0),(0,1)] \subset \Ker(\mathcal{L})^{\perp}=[(h',ch'')]^\perp,$$ 
such that $\mathcal{L}_{\Pi}\big |_{S^\perp}=\mathcal{L}$. On the other hand, corresponding to the constrained set $S$, we define the matrix
 
 \begin{equation}\label{matrixD}
 	D=\left[\begin{array}{llll}(\mathcal{L}^{-1}(1,0),(1,0))_{\mathbb{L}^2_{per}}& & (\mathcal{L}^{-1}(1,0),(0,1))_{\mathbb{L}^2_{per}}\\\\
 		(\mathcal{L}^{-1}(1,0),(0,1))_{\mathbb{L}^2_{per}}& & (\mathcal{L}^{-1}(0,1),(0,1))_{\mathbb{L}^2_{per}}\end{array}\right].
 	\end{equation}
Since $\mathcal{L}(0,1)=(0,1)$ and $(1,1)=(1,0)+(0,1)$, we have $$\mathcal{L}^{-1}(1,0)=\mathcal{L}^{-1}(1,1)-\mathcal{L}^{-1}(0,1)=\mathcal{L}^{-1}(1,1)-(0,1)$$ and
 \begin{equation}\label{D}
\left( \mathcal{L}^{-1}(1,0), (1,0) \right)_{\mathbb{L}^2_{per}}
= ( \mathcal{L}_1^{-1} 1,1)_{L^2_{per}}.
\end{equation}
Furthermore,
\begin{equation}\label{D12}
		\left( \mathcal{L}^{-1}(1,0), (0,1) \right)_{\mathbb{L}^2_{per}}=\left(\mathcal{L}^{-1}(1,1),(0,1)\right)_{\mathbb{L}^2_{per}}-\left(\mathcal{L}^{-1}(0,1),(0,1)
		\right)_{\mathbb{L}^2_{per}}=L-L=0\end{equation}
and 
\begin{equation}\label{D123}
	\left( \mathcal{L}^{-1}(0,1), (0,1) \right)_{\mathbb{L}^2_{per}}=\left((0,1),(0,1)\right)_{\mathbb{L}^2_{per}}=L.\end{equation}
According to \eqref{D}, \eqref{D12} and \eqref{D123}, it follows that $D$ can be expressed as $$D=\left[\begin{array}{cc}D_1&  0\\
	0&  L\end{array}\right],$$ where $D_1=(\mathcal{L}_{1}^{-1}1,1)_{L_{per}^2}$. The computation of $D_1$ requires finding, since $\Ker(\mathcal{L}_1)=[h']$, an element $\tilde{f}\in H_{per}^2$ satisfying $\mathcal{L}_1\tilde{f}=1$. 
	
	To obtain an appropriate periodic function $\tilde{f}$, we take the following steps: let us consider the first and the fifth eigenvalues of $\mathcal{L}_1$, and their corresponding eigenfunctions, given respectively by 
	\begin{align*}
		\lambda_0=\frac{1+k^2-2\sqrt{1-k^2+k^4}}{1+k^2}, \quad f_0(x)=1-[1+k^2-\sqrt{1-k^2+k^4}]\sn^2(bx;k),
	\end{align*}
	and 
	\begin{align*}
		\lambda_4=\frac{1+k^2+2\sqrt{1-k^2+k^4}}{1+k^2}, \quad f_4(x)=1-[1+k^2+\sqrt{1-k^2+k^4}]\sn^2(bx;k),
	\end{align*}where $ b =\frac{1}{\sqrt{\omega(1+k^2)}}=\frac{4K(k)}{L}$.
	Under these conditions, defining $B_1=(1+k^2+\sqrt{1-k^2+k^4})$ and $B_2 =-(1+k^2-\sqrt{1-k^2+k^4})$, we have that
	\begin{align}\label{eq1.autovalor}
		B_1f_0=&B_1-[(1+k^2)^2-(1-k^2+k^4)]\sn^2(bx;k)= B_1-3k^2\sn^2(bx;k)
	\end{align}
	and
	\begin{align}\label{eq2.autovalor}
		B_2f_4=B_2+[(1+k^2)^2-(1-k^2+k^4)]\sn^2(bx;k)=B_2+3k^2\sn^2(bx;k).
	\end{align}
It follows from \eqref{eq1.autovalor} and \eqref{eq2.autovalor} that 
\begin{align}\label{B1B2}
B_1f_0+B_2f_4=2\sqrt{1-k^2+k^4}. 
\end{align}
\indent Therefore, by the definition of $\lambda_0$ and $\lambda_4$, and using the equality $(\ref{B1B2})$, we deduce 
\begin{equation}\label{eq3.autovalor}
\mathcal{L}_1(\lambda_4 B_1f_0+\lambda_0B_2f_4)= 2\lambda_0\lambda_4\sqrt{1-k^2+k^4}.
\end{equation}
 Since $\lambda_0\lambda_4\sqrt{1-k^2+k^4}$ is nonzero for $k\in (0,1)$, we obtain by \eqref{eq3.autovalor} that $\mathcal{L}_1\tilde{f}=1$, or equivalently, $\tilde{f}=\mathcal{L}_1^{-1}1$, where $\tilde{f}\in  H_{per}^2$ is defined by $$\tilde{f}=\frac{ 1 }{2\lambda_0\lambda_4\sqrt{1-k^2+k^4}}(	\lambda_4 B_1f_0+\lambda_0B_2f_4).$$
Hence, \begin{align}\label{eq4.autovalor}
D_1=(\mathcal{L}_{1}^{-1}1,1)_{L_{per}^2}=(\tilde{f},1)_{L_{per}^2}=\frac{\lambda_4 B_1(f_0,1)_{L_{per}^2}+\lambda_0B_2(f_4,1)_{L_{per}^2}}{2\lambda_0\lambda_4\sqrt{1-k^2+k^4} }.
\end{align}
On the other hand, using the periodicity of the even function $\displaystyle\dn(u+2K(k);k)= \dn(u;k)$ and \cite[Formula 110.07]{byrd}, we obtain the relation  
$\displaystyle
	 	(\sn^2(bx;k),1)_{L_{per}^2}
	 	=\frac{4(K(k)-E(k))}{bk^2},
$
 where $E(k)=\displaystyle\int_0^{\frac{\pi}{2}}\sqrt{1-k^2\sin(\theta)^2}d\theta$ is the complete elliptic integral of the second kind.
 \\ Thus, it follows that
\begin{align}
	\label{eq5.autovalor}
	(f_0,1)_{L_{per}^2}=L-\frac{4B_2}{b}\frac{(E(k)-K(k))}{k^2},
\end{align}
and 
\begin{align}
	\label{eq6.autovalor}
	(f_4,1)_{L_{per}^2}=L+\frac{4B_1}{b}\frac{(E(k)-K(k))}{k^2}.
\end{align}
 Therefore, by \eqref{eq5.autovalor}, \eqref{eq6.autovalor}, and the fact that $b=\frac{4K(k)}{L}$, we obtain
  \begin{equation}\label{eq8.autovalor}
 	\lambda_4B_1	(f_0,1)_{L_{per}^2}+\lambda_0 B_2	(f_4,1)_{L_{per}^2}= 6L\sqrt{1-k^2+k^4}+ \frac{12L \sqrt{1-k^2+k^4}}{(1+k^2)}\frac{(E(k)-K(k))}{ K(k)}.
 \end{equation} 
\indent Next, since $2\lambda_0\lambda_4=\frac{-6(1-k^2)^2}{(1+k^2)^2}$, we conclude from \eqref{eq4.autovalor} and \eqref{eq8.autovalor} that
 \begin{equation}\label{eq9.autovalor}\begin{array}{lllll}
 \displaystyle	D_1=(\mathcal{L}_{1}^{-1}1,1)_{L_{per}^2}&=&\displaystyle \frac{6L}{2\lambda_0\lambda_4} + \frac{12L  }{(1+k^2)}\frac{(E(k)-K(k))}{ K(k)2\lambda_0\lambda_4 }\\\\&=&\displaystyle -\frac{L(1+k^2)^2}{(1-k^2)^2} - \frac{2L (1+k^2) }{(1-k^2)^2}\frac{(E(k)-K(k))}{K(k)}\\\\&=&\displaystyle\frac{-L(1+k^2)}{(1-k^2)^2}\bigg\{ (1+k^2)+2 \frac{(E(k)-K(k))}{ K(k)} \biggl \}.\end{array}
 \end{equation}
 By \cite[Formula 710.00 and 710.02]{byrd}, it follows that
  \begin{equation}\label{eq10.autovalor}
 	  (1-k^2)K(k)<E(k)<K(k), \ \mbox{for all}\ k\in (0,1).
 \end{equation}
 Using \eqref{eq9.autovalor} and \eqref{eq10.autovalor} we obtain $	D_1=(\mathcal{L}_{1}^{-1}1,1)_{L_{per}^2}<0$. 
 
%

\indent Therefore, using $(\ref{indexformula12})$ and $(\ref{indexformula123})$, it follows that
\begin{equation*}
		{\rm{n}}(\mathcal{L}_{1\Pi}) = 	{\rm{n}}(\mathcal{L}_{1})-{\rm{n}}_0-{\rm{z}}_0=1 - 1 - 0 = 0 \quad \mbox{and}\quad  	{\rm{z}}(\mathcal{L}_{1\Pi})={\rm{z}}(\mathcal{L}_{1\ })+{\rm{z}}_0=1+0=1.
\end{equation*}
Associated with the full linear operator $\mathcal{L}_{\Pi}$, we have demonstrated that
\begin{equation*}
	 {\rm{n}}(\mathcal{L}_{\Pi})=	{\rm{n}}(\mathcal{L})-{\rm{n}}_0-{\rm{z}}_0=1-1-0=0  \quad \mbox{and}\quad  {\rm{z}}(\mathcal{L}_{\Pi})={\rm{z}}(\mathcal{L})+{\rm{z}}_0=1+0=1,
\end{equation*}  as requested.
\end{proof}


\section{Orbital stability of periodic waves for the  $\phi^4-$equation}\label{section5}

The goal of this section is to establish a result of orbital stability based on the  theory contained in \cite[Section 4]{Natali2015} (see also \cite{grillakis1}) for the periodic  traveling wave solution $h$ in \eqref{existSG1}  for the  $\phi^4-$equation in \eqref{KF2}.  Consider the restricted energy space $Y= H_{per,m}^1 \times L_{per,m}^2$. It is well known that \eqref{KF2} is invariant by translations. Thus, we can define for $x,s \in \mathbb{R}$ and $U=(u,v) \in Y$   the action
\begin{equation*}
\mathrm{T}_sU(x)=(u(x+s), v(x+s)). 
\end{equation*}

Next we recall the definition of the orbital stability in this context.

\begin{definition}[Orbital Stability]\label{stadef}
The periodic wave $(h,ch')$ is said to be orbitally stable in $Y$ if for all $\varepsilon>0$ there exists $\delta>0$ with the following property: if
$$
\|(\phi_0, \phi_1)-(h, ch')\|_{Y}<\delta,
$$ 
then the weak solution $\Phi=(\phi,\phi_t)$ of \eqref{KF2} with the initial condition $\Phi(0)=(\phi_0,\phi_1)\in Y$ exists for all $t\geq0$,  and it satisfies
$$
\inf_{s \in \mathbb{R}}\big\|\Phi(t)-\mathrm{T}_s (h,ch') \big\|_Y<\varepsilon,
$$
$\mbox{for all}\ t\geq0.$ Otherwise,  $(h,ch')$ is said to be  orbitally unstable. In particular, this would happen in the case of solutions that blow up in finite time.

\end{definition}

We now prove Theorem $\ref{stabthm}$ as an immediate consequence of the following proposition.

\begin{proposition}\label{propstab}
Let $L\in (0,2\pi)$ be fixed. There exists a constant $C>0$ such that 
\begin{equation}\label{positiveL12}
(\mathcal{L}(p,q),(p,q))_{\mathbb{L}_{per}^2}=(\mathcal{L}_{\Pi}(p,q),(p,q))_{\mathbb{L}_{per}^2}\geq C||(p,q)||_{\mathbb{L}_{per}^2}^2,
\end{equation}
for all $(p,q)\in H_{per,m}^2\times H_{per,m}^1$ such that $((p,q),(h',ch''))_{\mathbb{L}_{per}^2}=0$. In particular, the statement of Theorem $\ref{stabthm}$ is valid. Moreover, we have that $d''(c)<0$, where $d(c)=\mathcal{E}(h,ch')-c\mathcal{F}(h,ch')$.
\end{proposition}
\begin{proof}

	The first part is an immediate consequence of \cite[page 278]{kato} and the fact that $\mathcal{L}_{\Pi}$ does not have negative eigenvalues. The estimate in $(\ref{positiveL12})$ and the arguments in \cite[Section 4]{Natali2015} are sufficient to conclude the statement of Theorem $\ref{stabthm}$.\\
	\indent  We prove that $d''(c)<0$ is verified without using the arguments in \cite[Subsection 4.2]{loreno}. Indeed, using Proposition $\ref{curveeqSG1}$, we can derive equation $(\ref{travSG1})$ with respect to $c\in(-1,1)$ to get $\mathcal{L}_1\left(\frac{\partial h}{\partial c}\right)=-2ch''$. Thus, since $\frac{\partial h}{\partial c}$ is an odd function for all $c\in(-1,1)$, we obtain that 
\begin{equation}\label{calcL}\mathcal{L}_{1\Pi}\left(\frac{\partial h}{\partial c}\right)=\mathcal{L}_1\left(\frac{\partial h}{\partial c}\right)-\frac{3}{L}\int_0^L(h(x))^2\frac{\partial h(x)}{\partial c}dx=
\mathcal{L}_1\left(\frac{\partial h}{\partial c}\right)=-2ch''.
\end{equation}
\indent Next, a simple calculation using $(\ref{calcL})$ also gives
\begin{equation}\label{calcL1}
\mathcal{L}\left(\frac{\partial h}{\partial c},\frac{\partial}{\partial c}(ch')\right)
=(-ch'',h')=\mathcal{F}'(h,ch'),
\end{equation}
where $\mathcal{F}'$ indicates the Fr\'echet derivative of $\mathcal{F}$ defined in $(\ref{F})$. Similarly as in $(\ref{calcL})$, we also have 
\begin{equation}\label{calcL2}
\mathcal{L}_{\Pi}\left(\frac{\partial h}{\partial c},\frac{\partial}{\partial c}(ch')\right)
=(-ch'',h')=\mathcal{F}'(h,ch').
\end{equation}
\indent On the other hand, since $d(c)=\mathcal{E}(h,ch')-c\mathcal{F}(h,ch')$ and $(h,ch')$ is a critical point of $\mathcal{G}(\phi,\psi)=\mathcal{E}(\phi,\psi)-c\mathcal{F}(\phi,\psi)$, we obtain that $d'(c)=-\mathcal{F}(h,ch')$. Therefore,
\begin{equation}\label{calcL3}\begin{array}{lllll}
d''(c)&=&\displaystyle\left(-\mathcal{F}'(h,ch'),\left(\frac{\partial h}{\partial c},\frac{\partial}{\partial c}(ch')\right)\right)_{\mathbb{L}_{per}^2}\\\\
&=&\displaystyle\left(-\mathcal{L}_{\Pi}\left(\frac{\partial h}{\partial c},\frac{\partial}{\partial c}(ch')\right),\left(\frac{\partial h}{\partial c},\frac{\partial}{\partial c}(ch')\right)\right)_{\mathbb{L}_{per}^2}\\\\
&=&\displaystyle-\frac{\partial }{\partial c}\mathcal{F}(h,ch')=-\frac{\partial }{\partial c}\left(c\int_0^L(h'(x))^2dx\right).
\end{array}
\end{equation}
\indent Next, in $\mathbb{L}_{per,m}^2$ consider the decomposition
\begin{equation}\label{calcL4}\left(\frac{\partial h}{\partial c},\frac{\partial}{\partial c}(ch')\right)=b_0(h',ch'')+(P,Q),
\end{equation}
where $(P,Q)\in\mathbb{H}_{per,m}^2$ is an element of the positive subspace of $\mathbb{L}_{per,m}^2$, that is, an element that satisfies
 \begin{equation}\label{calcL5}
(\mathcal{L}_{\Pi}(P,Q),(P,Q))_{\mathbb{L}_{per}^2}\geq C||(P,Q)||_{\mathbb{L}_{per}^2}^2,
\end{equation}
for some constant $C>0$. Thus, we have by $(\ref{calcL3})$, $(\ref{calcL4})$, $(\ref{calcL5})$, and some rudimentary calculations
$$
-d''(c)=\displaystyle\left(\mathcal{L}_{\Pi}\left(\frac{\partial h}{\partial c},\frac{\partial }{\partial c}(ch')\right),\left(\frac{\partial h}{\partial c},\frac{\partial }{\partial c}(ch')\right)\right)_{\mathbb{L}_{per}^2}=(\mathcal{L}_{\Pi}(P,Q),(P,Q))_{\mathbb{L}_{per}^2}>0.
$$
This last fact finishes the proof of the proposition.

\end{proof}

\begin{remark}
	The case $\omega < 0$ can be studied. This corresponds to $c > 1$ or $c < -1$, and equation $(\ref{KF3})$ becomes \begin{equation}\label{KF4}-\tau h''+h-h^3=0,\end{equation} where $\tau = -\omega > 0$. This ODE admits two families of periodic wave solutions with cnoidal and dnoidal profiles. It is well known that the cnoidal solution to equation $(\ref{KF4})$ has the zero-mean property, while the dnoidal solution is strictly positive. The problem of orbital instability for cnoidal solutions in the space $H_{per,m}^1\times L_{per,m}^2$ was addressed in \cite{loreno1}. The orbital instability of positive dnoidal waves can be obtained by combining the results in \cite{palacios} (for the case of superluminal waves) with \cite{grillakis1}. In fact, using \cite[Remark 3.2]{palacios}, the period mapping $\mathcal{T}$ associated with the dnoidal waves is strictly increasing in terms of the energy levels. Therefore, the linearized operator $\mathcal{L}$ in $(\ref{matrixop313})$ for dnoidal solutions  has only one negative eigenvalue, which is simple, and zero is a simple eigenvalue associated with the eigenfunction $(h',ch'')$. On the other hand, since $(h,ch')$ is also a critical point of $\mathcal{G}(\phi,\psi)=\mathcal{E}(\phi,\psi)-c\mathcal{F}(\phi,\psi)$, we obtain that $d'(c)=-\mathcal{F}(h,ch')$. Therefore, as in $(\ref{calcL3})$, we obtain $$d''(c)=-\frac{d}{dc}\mathcal{F}(h,ch')=-\frac{d}{dc}\left(c\int_0^L(h'(x))^2dx\right).$$
	A direct computation based on the properties of elliptic functions reveals that $d''(c)<0$ for dnoidal waves when $|c|>1$. Therefore, the Instability Theorem in \cite{grillakis1} establishes that the dnoidal wave $(h,ch')$ is orbitally unstable in the sense of Definition $\ref{stadef}$.
	
\end{remark}

\section{Concluding Remarks} In this paper, we present a different approach to studying the orbital stability of periodic snoidal waves for the well-known $\phi^4-$equation. It is important to mention that the $\phi^4-$equation is set within the Klein-Gordon regime, and that the results in \cite{loreno} and \cite{palacios} are, in some sense, consistent with the celebrated paper \cite{grillakis1} (see Section 5, Example A), where solitary waves are expected to be unstable in the full energy space. The stationary case $c=0$ produces orbitally stable snoidal waves in $H_{per,odd}^1\times L_{per,odd}^2$, as reported in \cite{loreno} and \cite{palacios}, and this is in accordance with the results determined in \cite{martel}. Here, we present the orbital stability in a new energy space $H_{per,m}^1\times L_{per,m}^2$, consisting by periodic functions with zero mean. The reason for this is that the zero-mean condition eliminates negative directions associated with the projected linearized operator restricted to zero-mean perturbations, allowing a refined spectral analysis and providing the orbital stability. Another important feature of our work is its adaptability to other Klein-Gordon equations, such as the sine-Gordon and sinh-Gordon equations (see \cite{Natali2011} for further details). In these cases, global solutions can only be proven in the space $H_{per,m}^1\times L_{per,m}^2$ using the Poincaré-Wirtinger inequality in $(\ref{PW1})$ and an argument similar to that in Remark $\ref{gwp}$. Thus, our results offer a new perspective on the study of periodic waves with the zero-mean property in the context of Klein-Gordon-type equations.


\end{document}